\documentclass[reqno,12pt,letterpaper]{amsart}
\usepackage[proof]{macros}

\usepackage{enumerate}
\usepackage{subfig,caption}
\usepackage{graphicx}				
\usepackage{amssymb}
\usepackage{amsmath}
\usepackage{amsthm,amsfonts,bm,bbm}
\usepackage{mathtools}
\mathtoolsset{showonlyrefs=true}
\usepackage{mathrsfs}
\usepackage{xcolor}
\usepackage{enumerate,enumitem}
\usepackage{leftidx}
\numberwithin{equation}{section}
\usepackage{marginnote}
\title[Weakly elliptic damping]{Weakly elliptic damping gives sharp decay}

\author{Lassi Paunonen}
\address{Mathematics and Statistics, Faculty of Information Technology and Communication Sciences, Tampere University, PO Box 692, 33101 Tampere, Finland}
\email[Lassi Paunonen]{lassi.paunonen@tuni.fi}

\author{Nicolas Vanspranghe}
\address{Mathematics and Statistics, Faculty of Information Technology and Communication Sciences, Tampere University, PO Box 692, 33101 Tampere, Finland}
\email[Nicolas Vanspranghe]{nicolas.vanspranghe@tuni.fi}

\author{Ruoyu P. T. Wang}
\address{Department of Mathematics, University College London, London, WC1H 0AY, United Kingdom}
\email[Ruoyu~P.~T.~Wang]{ruoyu.wang@ucl.ac.uk}

\date{}

\newcommand{\essinf}{\mathop {\rm ess\,inf}}
\newcommand{\esssup}{\mathop {\rm ess\,sup}}

\renewcommand{\supp}{\text{supp }}

\newcommand{\Ac}{\mathcal{A}}

\newcommand{\Hc}{\mathcal{H}}

\newcommand{\abs}[1]{\left|#1\right|}

\newcommand{\cim}{\operatorname{Im}}
\newcommand{\cre}{\operatorname{Re}}

\newcommand{\Hcd}{\dot{\Hc}}

\renewcommand{\ker}{\operatorname{Ker}}
\newcommand{\img}{\operatorname{Im}}

\begin{document}

\begin{abstract}
We prove that weakly elliptic damping gives sharp energy decay for the abstract damped wave semigroup, where the damping is not in the functional calculus. In this case, there is no overdamping. We show applications in linearised water waves and Kelvin--Voigt damping.
\end{abstract}

\maketitle

\section{Introduction}
\subsection{Motivation}\label{s11}
Let $\Delta\ge 0$ be the Laplace--Beltrami operator on a compact manifold $M$ of dimension $d\ge 1$ without boundary. There is recent interest in the study of decay of the damped water wave equation, linearised via paradifferential diagonalisation,
\begin{gather}\label{eq:5}
(\partial_t^2+\Delta^{\frac{1}{2}}_x+\Delta_x^{\frac{1}{4}} a(x) \Delta_x^{\frac{1}{4}}\partial_t)u(t,x)=0,\\
u(0,x)\in W^{\frac{1}{2}, 2}(M),\ \partial_t u(0, x)\in L^2(M)
\end{gather}
describes the evolution of a fluid interface in the gravity-capillary water wave system subject to an external pressure, studied in \cite{abz11,ala17b,ala18,abh18,amw23,kw23}. We define the energy of the solution to \eqref{eq:5} by
\begin{equation}
E(u,t)=\|\partial_t u\|_{L^2(M)}^2+\|\Delta^{\frac{1}{4}}u\|_{L^2(M)}^2.
\end{equation}
We want to understand the decay of energy, when the dissipation coefficient $a(x)\in L^\infty(M)$ may vanish on a measure zero set in $M$. Here the damping term $\Delta_x^{\frac{1}{4}} a(x) \Delta_x^{\frac{1}{4}}$, though not relatively compact, may still have some weak elliptic properties. This motivates us to study how elliptic damping gives sharp energy decay rates in generalised semigroup setting. As a corollary, in Example \ref{p:water3} we prove that the energy of \eqref{eq:5} decays exponentially when $a(x)$ degenerates fast near its zeros.
\begin{corollary}
Assume $(a(x))^{-1}\in L^{p}(M)$ for $p\in (d,\infty)$. Then there is $C>0$ that
\begin{equation}
E(u, t)\le e^{-Ct}E(u,0),
\end{equation}
uniformly in $t>0$ and $u$ satisfying \eqref{eq:5}. For $p\in (1,d]$, we have
\begin{equation}
E(u,t)^{\frac{1}{2}}\le C\langle t\rangle^{-\frac{p}{d-p}}(\|u(0,x)\|_{W^{1,2}(M)}+\|\partial_t u(0,x)\|_{W^{\frac{1}{2},2}(M)})
\end{equation}
uniformly in $t>0$ and $u$ solving \eqref{eq:5} with initial data in $W^{1,2}(M)\times W^{\frac{1}{2},2}(M)$. 
\end{corollary}

\subsection{Introduction}
Let $H=H_0$ be an infinite-dimensional Hilbert space and $P:H_1\rightarrow H$ be a nonnegative self-adjoint operator with compact resolvent, defined on $H_1$, a dense subspace of $H$. The operator $P$ admits a spectral resolution and a functional calculus
\begin{equation}
P u=\int_0^\infty\rho^2\ dE_\rho(u), \ f(P)u=\int_0^{\infty} f(\rho^2)~dE_\rho(u),
\end{equation}
where $E_\rho$ is a projection-valued measure on $H$ and $\supp E_\rho\subset [0, \infty)$, and $f$ is a Borel measurable function on $[0,\infty)$ that formally yields an operator $f(P)$. For $s\in\mathbb{R}$, define the scaling operators and the interpolation spaces via
\begin{equation}\label{Lambdadef}
\Lambda^{s}u=\int_0^\infty (1+\rho^2)^{s}\ dE_\rho(u), \ H_s=\Lambda^{-s}(H_0).
\end{equation}
Those operators $\Lambda^{-s}: H\rightarrow H_s$ are bounded from above and below, and they commute with $P$. For $s>0$, $H_{-s}$ is isomorphic to the dual space of $H_s$ with respect to $H$. 

Let the observation space $Y$ be a Hilbert space. We will consider damping of the form $Q^*Q$, where the control operator $Q^*\in \mathcal{L}(Y, H_{-\frac{1}{2}})$ and the observation operator $Q\in \mathcal{L}(H_{\frac{1}{2}}, Y)$. Note that $Q^*Q$ is not necessarily a bounded operator on $H$. We consider an abstract damped second-order evolution equation:
\begin{equation}
(\partial_t^2+P+Q^*Q\partial_t)u=0.
\end{equation}
It can be written as a first-order evolution system:
\begin{equation}
\partial_t\begin{pmatrix}
u\\\partial_t u
\end{pmatrix}=\mathcal{A}\begin{pmatrix}
u\\\partial_t u
\end{pmatrix}, \ \mathcal{A}=\begin{pmatrix}
0 & 1\\
-P & -Q^*Q
\end{pmatrix}. 
\end{equation}
Here $\Ac$, defined on $\{(u, v)\in H_{\frac{1}{2}}\times H_{\frac{1}{2}}: Pu+Q^*Qv\in H\}$, generates a strongly continuous semigroup $e^{t\Ac}$ on $\mathcal{H}=H_{\frac{1}{2}}\times H$. See \cite{kw23,cpsst19} for further details.

\subsection{Main results}
The goal of this note is to understand the stability, that is the norm of $(\mathcal{A}+i\lambda)^{-1}$, and thus decay, of $e^{t\Ac}$ on the energy space $\Hcd=\Hc/\ker\Ac$, when $Q^*Q$ satisfies some ellipticity conditions but is not by itself in the functional calculus of $P$. We define the weak ellipticity and boundedness:
\begin{definition}[$m$-ellipticity]\label{d:mell}
Let $m(\lambda)$ be a positive continuous function on $(0,\infty)_\lambda$. We say the observation operator $Q: H_{\frac{1}{2}}\rightarrow Y$ is \emph{$m(\lambda)$-elliptic}, if {for some $\chi\in C^0([0, \infty))$ with $\chi(1)>0$,} there exist and $C, \lambda_0, N>0$ that 
\begin{equation}\label{eq:7}
m(\lambda)\|\chi(\lambda^{-2}P) u\|_H^2\le C\|Q u\|_Y^2+{o(\min\{m(\lambda)\lambda^{4N}, \lambda\})\|\Lambda^{-N}u\|_H^2}
\end{equation}
uniformly for all $u\in H_{\frac{1}{2}}$ and all $\lambda\in \mathbb{R}$ with $\lambda\ge \lambda_0$. 
\end{definition}
\begin{definition}[$m$-boundedness]
Let $m(\lambda)$ be a positive continuous function on $(0,\infty)_\lambda$. We say the observation operator $Q: H_{\frac{1}{2}}\rightarrow Y$ is \emph{$m(\lambda)$-bounded}, if {for some $\chi\in C^0([0, \infty))$ with $\chi(1)>0$,} there is $C>0$ such that
\begin{equation}
\|(1+\lambda^{-2}P)^{-\frac{1}{2}}Q^*Q\chi(\lambda^{-2}P)u\|_H^2\le Cm(\lambda)\|(1+\lambda^{-2}P)^{\frac{1}{2}} u\|_H^2,
\end{equation}
for all $u\in H_{\frac{1}{2}}$. 
\end{definition}

\begin{remark}
\label{r:mellbounded}
\begin{enumerate}[wide]
\item For readers familiar with semiclassical analysis, $m$-ellipticity (or $m$-boundedness) heuristically means $(m(h^{-1}))^{-\frac{1}{2}}Q$ (or $(m(h^{-1}))^{-1}Q^*Q$) being semiclassically elliptic (or bounded) over $\{h^2\rho^2=1\}$, the characteristic variety of $h^2P-1$. 
\item Consider two positive continuous functions $m_-(\lambda)\le m_+(\lambda)$ on $(0,\infty)_\lambda$. If $Q$ is $m_+$-elliptic, then $Q$ is also $m_-$-elliptic. If $Q$ is $m_-$-bounded, then $Q$ is also $m_+$-bounded. 
\item Consider two observation operators $Q_\pm: H_{\frac{1}{2}}\rightarrow Y_\pm$ such that $\|Q_- u\|_{Y_{-}}^{2}\le C\|Q_+ u\|_{Y_{+}}^{2}$ uniformly for $u\in H_{\frac{1}{2}}$. If $Q_-$ is $m$-elliptic, then $Q_+$ is also $m$-elliptic. 
\item Any observation operator $Q$ bounded from $H_{\frac{1}{2}}$ to $Y$ is a priori $\lambda^2$-bounded. 
\end{enumerate}
\end{remark}

For classically elliptic and bounded operators, the $m$-ellipticity and $m$-boundedness are easy to verify via Theorem \ref{p:classell} proved at the end of this note:
{\begin{example}[Linearised water waves]\label{p:water1}
In the setting of \S\ref{s11}, $P=\Delta$, $H=L^2(M)$, $H_s=W^{s, 2}(M)$ are the Sobolev spaces of order $s$ for $s\in\mathbb{R}$, $Q=\sqrt{a(x)}\Delta^{\frac{1}{4}}: H_{\frac{1}{2}}\rightarrow Y=H$. When $a(x)\in L^\infty(M)$ is bounded from above and below by positive constants, $Q$ is classically elliptic with respect to $\Lambda^{\frac{1}{2}}$:
\begin{equation}
\|\sqrt{a(x)}\Delta^{\frac{1}{4}} u\|_{L^2(M)}\ge C^{-1}\|\Lambda^{\frac{1}{2}}u\|_{L^2(M)}-C\|u\|_{L^2(M)}
\end{equation}
and Theorem \ref{p:classell} implies $Q$ is $\lambda^{2}$-elliptic and $\lambda^{2}$-bounded. 
\end{example}}

Our result is bifold. The first result is that $m$-elliptic damping gives an upper bound for semigroup stability. 
\begin{theorem}[Weak ellipticity gives stability]\label{t:decay}
Let $m(\lambda)$ be positive continuous function and let $Q$ be $m(\lambda)$-elliptic. Then there are $C, \lambda_0>0$ such that
\begin{equation}
\|(\mathcal{A}+i\lambda)^{-1}\|_{\mathcal{L}(\mathcal{H})}\le C \max \left\{\frac{1}{m(\abs{\lambda})}, 1\right\} 
\end{equation}
uniformly for all $\abs{\lambda}\ge \lambda_0$, $\lambda\in \mathbb{R}$. 
\end{theorem}

When the unique continuation hypothesis holds (that is, $(P-\lambda^2)u=0$ implies $Qu\neq 0$ for $\lambda>0$), and 
$m(\lambda)$ is chosen to be $1$, or the reciprocal of a function of positive increase, for example, $\lambda^{-s}$ or $e^{-s\lambda}$ for some $s\ge 0$, one can use \cite[Lemma 3.10]{kw23} (based on semigroup equivalence results in \cite{bt10,rss19}), to turn the stability results into exponential, polynomial and logarithmic energy decay for $e^{t\mathcal{A}}$ respectively.

The second result is that $m$-bounded damping gives a lower bound for the semigroup stability, which is asymptotic to the upper bound found in Theorem \ref{t:decay}. 
\begin{theorem}[Weak boundedness gives sharpness]\label{t:sharp}
Let $m(\lambda)$ be positive continuous function and let $Q$ be $m(\lambda)$-bounded. Then there exist a sequence of $\lambda_k\rightarrow \infty$ and $C>0$,
\begin{equation}
\|(\mathcal{A}+i\lambda_k)^{-1}\|_{\mathcal{L}(\mathcal{H})}\ge \frac{1}{Cm(\abs{\lambda_k})}.
\end{equation}
\end{theorem}
It implies that when $Q$ is $m$-bounded, it is not possible to improve the bound in Theorem \ref{t:decay} to $o(1/M(\abs{\lambda}))$. Thus, we have proved that a $m$-elliptic and $m$-bounded damping gives sharp semigroup stability. We have a handy corollary below to show energy decay:

\begin{corollary}[Weak ellipticity gives decay]\label{t3}
Let $Q: H_{\frac{1}{2}}\rightarrow Y$ be bounded from below by $\Lambda^{s}$ for some $s\in (-\infty, 0)$, that is, there exists $C>0$ such that
\begin{equation}\label{eq:4}
\|\Lambda^{s} u\|_{H}\le C \|Qu\|_{Y}
\end{equation}
uniformly for $u\in H_{\frac{1}{2}}/\ker P$. Then there is $C>0$ such that for all $t\ge 0$, 
\begin{equation}
\|e^{t\Ac}\Ac^{-1}\|_{\Hc\rightarrow\Hcd}\le C\langle t\rangle^{\frac{1}{4s}}.
\end{equation}
When \eqref{eq:4} holds with $s=0$, then there is $C>0$ such that
\begin{equation}
\|e^{t\Ac}\|_{\Hcd\rightarrow \Hcd}\le e^{-Ct}. 
\end{equation}
\end{corollary}
\begin{remark}[Strong monotonicity]
Corollary \ref{t3} is interpreted that damping given by $Q$ larger than $\Lambda^s$ must give at least the same decay rate as that given by $\Lambda^s$. In particular, in this case there is no overdamping, contrary to the general case of weak monotonicity discussed in \cite[\S 2.2]{kw23} (see also \cite{aln14,sta17,kle19b,dk20,kw22}). 
\end{remark}
Theorems \ref{t:decay}, \ref{t:sharp} {and Corollary \ref{t3}} constitute improvements over \cite[Theorems 12.1 and 12.2]{DelPat21}, where $Q^*Q=f(P)$ for some nonnegative continuous function $f$. Such an assumption is often too strong and restrictive. Our results do not require $Q^*Q$ to be in the functional calculus of $P$ and also accommodates compact errors. {The compact errors are natural in the studies of elliptic estimates, and accommodating them in our theorems allows us to obtain new stability results: see Example \ref{p:egpdo}. {Corollary \ref{t3}} also generalises the result of \cite{LiuZha15}, the authors of which studied the case $\ker P=0$ without allowing compact errors in \eqref{eq:7}. The separation of $\ker P$ from $P\ge 0$ is not trivial, when $Q^*Q$ is no longer relatively compact.} Furthermore, in contrast to \cite{DelPat21,LiuZha15}, we are able to assess sharpness of the decay rates even for damping outside of the functional calculus.

Here we present concrete examples where we get new results. 

\begin{example}[Linearised water waves, bounded from below]\label{p:water2}
In the setting of \S\ref{s11} and Example \ref{p:water2}, assume $a(x)$ is bounded from above and below by positive constants. Note uniformly for all $u\in W^{\frac{1}{2}, 2}(M)/\operatorname{Span}\{1\}$, 
\begin{equation}
\|\Lambda^{\frac{1}{2}}u\|_{L^2(M)}\le C\|\sqrt{a(x)}\Delta^{\frac{1}{4}}u\|_{L^2(M)}.
\end{equation}
Corollary \ref{t3} implies for some $C>0$,
\begin{equation}
E(u,t)\le e^{-Ct}E(u, 0),
\end{equation}
for all $t>0$ and for all $u$ that solves \eqref{eq:5}.
\end{example}

\begin{example}[Linearised water waves, degenerate]\label{p:water3}
We now consider that in the setting of \S\ref{s11} that $a(x)\ge 0$ may vanish on a measure zero set. In order to control its degree of degeneracy, we assume $(a(x))^{-1}\in L^p(M)$ for $p\in (1,\infty)$: the larger $p$ is, the faster $a(x)$ vanishes near its zeros. Sobolev embedding implies that $(a(x))^{-\frac{1}{2}}: L^2(M)\rightarrow W^{-\frac{d}{2p},2}(M)$ is bounded (see for example, \cite[Lemma 2.17]{kw23}). Thus
\begin{equation}
\|a^{\frac{1}{2}}\Delta^{\frac{1}{4}}u\|_{L^2}\ge C^{-1}\|(1+\Delta)^{-\frac{d}{4p}}\Delta^{\frac{1}{4}}u\|_{L^2}\ge C^{-1}\|\Lambda^{-\frac{1}{4}(\frac{d}{p}-1)} u\|_{L^2}-C\|\Lambda^{-\frac{1}{4}(\frac{d}{p}-1)-1}u\|_{L^2}. 
\end{equation}
Thus $Q=a^{\frac{1}{2}}\Delta^{\frac{1}{4}}$ is classically elliptic with respect to $\Lambda^{-\frac{1}{4}(\frac{d}{p}-1)}$. Apply Theorem \ref{p:classell} to see $Q$ is $\lambda^{-(\frac{d}{p}-1)}$-elliptic. Furthermore, since $a$ only vanishes on a measure zero set, $a^{\frac{1}{2}}\Delta^{\frac{1}{4}}u=0$ implies $\Delta^{\frac{1}{4}}u=0$ on $M$, and thus $\Delta u=0$. This implies the unique continuation holds. Apply Theorem \ref{t:decay} with \cite[Lemma 3.10]{kw23} to see
\begin{enumerate}
\item  When $p\in [d, \infty)$ and $p\neq 1$, there is $C>0$ such that uniformly for all $t>0$,
\begin{equation}\label{eq:8}
E(u,t)\le e^{-Ct} E(u,0),
\end{equation}
for all $u$ solving \eqref{eq:5}. 
\item When $p\in (1,d)$, there is $C>0$ such that uniformly for all $t>0$, 
\begin{equation}
E(u,t)^{\frac{1}{2}}\le C\langle t\rangle^{-\frac{p}{d-p}}(\|u(0,x)\|_{W^{1,2}(M)}+\|\partial_t u(0,x)\|_{W^{\frac{1}{2},2}(M)})
\end{equation}
for all $u$ solving \eqref{eq:5} with initial data in $W^{1,2}(M)\times W^{\frac{1}{2},2}(M)$. 
\item When $p=1$ and $d\ge 2$, the Sobolev embedding works with $(a(x))^{-\frac{1}{2}}: L^2(M)\rightarrow W^{-\frac{d}{2p}-0, 2}(M)$. For each $\epsilon>0$,there is $C_\epsilon>0$ such that uniformly for all $t>0$, 
\begin{equation}
E(u,t)^{\frac{1}{2}}\le C_\epsilon\langle t\rangle^{-\frac{p}{d-p}+\epsilon}(\|u(0,x)\|_{W^{1,2}(M)}+\|\partial_t u(0,x)\|_{W^{\frac{1}{2},2}(M)})
\end{equation}
for all $u$ solving \eqref{eq:5} with initial data in $W^{1,2}(M)\times W^{\frac{1}{2},2}(M)$. 
\item When $p=1$ and $d=1$, for any $N>0$, there is $C_N>0$ such that uniformly for all $t>0$, 
\begin{equation}
E(u,t)^{\frac{1}{2}}\le C_N\langle t\rangle^{-N}(\|u(0,x)\|_{W^{1,2}(M)}+\|\partial_t u(0,x)\|_{W^{\frac{1}{2},2}(M)})
\end{equation}
for all $u$ solving \eqref{eq:5} with initial data in $W^{1,2}(M)\times W^{\frac{1}{2},2}(M)$. 
\end{enumerate}
As an example, consider $M=\mathbb{S}^1=[-1, 1]_{x}$ with endpoints identified. Consider the damping $a(x)=x^{2s}$ for $s\in [0, \frac{1}{2})$, then $(a(x))^{-1}\in L^{\frac{1}{2s}-0}(M)\subset L^{1+0}(M)$ and we always have exponential decay \eqref{eq:8}. Contextually, in \cite{amw23}, only polynomial decay has been shown with a smaller damping $Q=a^\frac{1}{2}$ that degenerates slowly near its zeros. 
\end{example}

The next few examples are devoted to the damped wave equations with unbounded, Kelvin--Voigt, and pseudodifferential damping. Let $M$ be a compact smooth manifold, $H=L^2(M)$, $P=\Delta\ge 0$ and $Q: W^{1,2}(M)\rightarrow Y$. The damped wave equation is 
\begin{gather}\label{eq:6}
(\partial_t^2+\Delta
% _x
 +Q^*Q \partial_t )u(t,x)=0,\\
u(0, x)\in {W^{1,2}(M)}, \partial_t u(0, x)\in {L^2(M)},
\end{gather}
The energy of the solution is
\begin{equation}
E(u,t)=\|\partial_t u\|_{L^2(M)}^2+\|\nabla u\|_{L^2(M)}^2. 
\end{equation}
\begin{example}[Damped wave equation]\label{eg:2}
Consider a function ${a(x)}\in L^\infty{(M)}$ with $\inf_{x\in M}{\abs{a(x)}}>0$, and let $Q=\sqrt{a(x)}\langle \Delta\rangle^{s} : W^{1,2}(M)\rightarrow L^2(M)$ for $s\in (-\infty,\frac{1}{2}]$. Here $\langle \cdot\rangle=(1+\abs{\cdot}^2)^{\frac{1}{2}}$. Note that $[Q^*Q, P]\neq 0$ and hence $Q^*Q$ is not in the functional calculus of $P$. Here $Q$ is a classically elliptic operator, and Theorem \ref{p:classell} implies $Q$ is $\lambda^{4s}$-elliptic. Corollary \ref{t3} further gives the energy decay rates:
\begin{enumerate}
\item  When $s\in [0, \frac{1}{2}]$, there is $C>0$ such that uniformly for all $t>0$,
\begin{equation}
E(u,t)\le e^{-Ct} E(u,0),
\end{equation}
for all $u$ solving \eqref{eq:6}. This rate still holds when we replace $\langle \Delta\rangle^{s}$ by {$\Delta^{s}$}.
\item When $s\in (-\infty, 0)$, there is $C>0$ such that uniformly for all $t>0$, 
\begin{equation}
E(u,t)^{\frac{1}{2}}\le C\langle t\rangle^{\frac{1}{4s}}(\|u(0,x)\|_{W^{2,2}(M)}+\|\partial_t u(0,x)\|_{W^{1,2}(M)})
\end{equation}
for all $u$ solving \eqref{eq:6} with 
initial data in $W^{2,2}(M)\times W^{1,2}(M)$. 
\end{enumerate}
When $s\in [0, \frac{1}{2}]$ and $a(x)\in L^\infty(M)$, Theorem \ref{p:classell} implies $Q$ is $\lambda^{4s}$-bounded, and Theorem \ref{t:sharp} implies the rate in (1) are optimal. When $s\in(-\infty, 0)$, if we further impose the regularity assumption to $a(x)\in W^{-2s{+0}, \infty}(M)$ being $(-2s+{0})$-Hölder, Theorem \ref{p:classell} implies
 $Q$ is $\lambda^{4s}$-bounded and the rate in (2) is optimal. 
\end{example}

\begin{example}[Kelvin--Voigt damping]
Let $A$ be a {bundle isomorphism} on $TM$ such that $A_x\in \operatorname{Iso}(T_xM)$, not necessarily continuous in $x\in M$. Assume 
\begin{equation}
\esssup_{x\in M}\|A_x\|_{\mathcal{L}(T_xM)}<\infty, \ \esssup_{x\in M}\|A_x^{-1}\|_{\mathcal{L}(T_xM)}<\infty.
\end{equation}
Then $Q=A \nabla: W^{1,2}(M)\rightarrow L^2(M, TM)$ is $\lambda^2$-bounded and $\lambda^2$-elliptic, where $\nabla$ is the gradient. Apply Corollary \ref{t3} to see
there is $C>0$ such that for all $t>0$,
\begin{equation}
E(u, t)\le e^{-Ct} E(u,0),
\end{equation}
for all $u$ solving \eqref{eq:6}. In the special case that $A_x (x,v)= (x,a(x)v)$ for $a(x)\in L^\infty(M)$ with $\essinf_{x\in M}\abs{a(x)}>0$, we recovered the exponential decay for the Kelvin--Voigt damping of $L^\infty$-regularity. Contextually, under additional $C^1$-regularity, it was shown in \cite{bur20,bs22} that the same exponential rate holds when $a(x)$ may vanish on some open sets but still satisfies the geometric control condition. {$L^\infty$-regularity in this case only gives optimal polynomial decay: see \cite{bs22b}.} 
\end{example}

\begin{example}[Pseudodifferential damping]\label{p:egpdo}
For $s\in (-\infty,\frac{1}{2}]$, consider $Q\in \Psi^{2s}(M)$, an (one-step polyhomogeneous) pseudodifferential operator of order $2s$. Then Theorem \ref{p:classell} implies $Q$ is $\lambda^{4s}$-bounded. If $Q$ is classically elliptic, that is, its principal symbol $\abs{\sigma_{2s}(Q)(x,\xi)}$ is uniformly bounded from below by $C^{-1}\langle \xi\rangle^{2s}$ on $\{\abs{\xi}\ge C\}$ for some $C>0$, then for any $N>0$, we have the elliptic estimate
\begin{equation}
\|Qu\|_{L^2(M)}\le C\|u\|_{W^{2s,2}(M)}+C_N\|u\|_{W^{-N, 2}(M)},
\end{equation}
uniformly for all $u\in W^{1,2}(M)$. Theorem \ref{p:classell} implies $Q$ is $\lambda^{4s}$-elliptic. Our Theorems \ref{t:decay} and \ref{t:sharp} below imply the sharp estimate $\|(\mathcal{A}+i\lambda)^{-1}\|_{\mathcal{L}(\mathcal{H})}\le C\max\{\lambda^{-4s},1\}$ for large $\lambda$. With extra unique continuation hypotheses on $Q$ explicitly, one can obtain energy decay results. Pseudodifferential damping models dissipation in anisotropic materials: the energy of waves is dissipated at different rates depending on the direction of propagation.  The case $s=0$ was studied in \cite{kk22}, and case $s\neq 0$ in \cite{kw23}. See further in Theorem \ref{p:classell} and references \cite{zwo12,dz19}.
\end{example}

\subsection{Acknowledgement}
The authors thank Jeffrey Galkowski and Jared Wunsch for discussions around the results. LP and NV are supported by the Research Council of Finland grant 349002. RPTW is supported by EPSRC grant EP/V001760/1.

\section{Proof}
Consider the semiclassical operator $P_h=h^2P-ihQ^*Q-1{:H_{\frac{1}{2}}\rightarrow H_{-\frac{1}{2}}}$, {where $h\in (0, h_0)$ for some small $h_0>0$. It is Fredholm: see \cite[Lemma 3.4]{kw23}. }In this section, we prove the upper bounds and lower bounds for
 $P_h^{-1}$ in Propositions \ref{p:semires} and \ref{p:semilwb}, respectively. We state and prove Theorem \ref{p:classell} at the end of this section. Throughout this section, we use the abbreviation $m=m(\lambda)=m(h^{-1})$. We start with an upper bound for $P_h^{-1}$.

\begin{proposition}[Semiclassical resolvent estimate]\label{p:semires}
Let $m$ be a positive continuous function on $(0,\infty)$ and let $Q$ be $m$-elliptic, that is, {for some $\chi\in C^0([0, \infty))$ with $\chi(1)>0$,} there exist $C, N>0$ such that
\begin{equation}\label{eq:semiell}
\|\chi(h^{2}P) u\|^2\le C(m(h^{-1}))^{-1}\|Q u\|_Y^2+{o(\min\{h^{-4N}, m(h^{-1})^{-1}h^{-1}\})}\|\Lambda^{-N}u\|^2,
\end{equation}
uniformly for all $u\in H_{\frac{1}{2}}$ and all small $h>0$. Then there is $C>0$ such that
\begin{equation}
\|P_h^{-1}\|_{\mathcal{L}(H)}\le C\max \left\{\frac{h^{-1}}{m(h^{-1})}, 1\right\} 
\end{equation} 
uniformly for all $h>0$ small.
\end{proposition}

To prove Proposition \ref{p:semires}, we need to estimate the compact error in \eqref{eq:semiell}.
%a lemma to estimate the compact error in \eqref{eq:semiell}.

\begin{lemma}[Compact error estimates]\label{p:cpterr}
For $N>0$, there exists $C>0$ such that 
% we have
\begin{equation}
\|\Lambda^{-N}u\|^2\le Ch^{4N}\|u\|^2+C\abs{\|hP^{\frac{1}{2}}u\|^2-\|u\|^2}. 
\end{equation}
uniformly for all $u\in H_{\frac{1}{2}}$ and all small $h>0$.
\end{lemma}
\begin{proof}
Consider that
\begin{equation}
\|\Lambda^{-N}u\|^2=\int_{\abs{h^2\rho^2-1}<\frac{1}{2}}(1+\rho^2)^{-2N}~\langle dE_\rho u, u\rangle +\int_{\abs{h^2\rho^2-1}\ge\frac{1}{2}}(1+\rho^2)^{-2N}~\langle dE_\rho u, u\rangle.
\end{equation}
We can estimate the second term
\begin{equation}
\int_{\abs{h^2\rho^2-1}\ge\frac{1}{2}}(1+\rho^2)^{-2N}~\langle dE_\rho u, u\rangle\le 2\int_{\abs{h^2\rho^2-1}\ge\frac{1}{2}} \abs{h^2\rho^2-1}~\langle dE_\rho u, u\rangle
\end{equation}
and therefore bound it by $C\abs{\|hP^{\frac{1}{2}}u\|^2-\|u\|^2}$. To bound the first term, 
we note that on $\{\abs{h^2\rho^2-1}<\frac{1}{2}\}$, we have $1+\rho^2>1+\frac{1}{2}h^{-2}$ and thus $(1+\rho^2)^{-2N}<(1+\frac{1}{2}h^{-2})^{-2N}\le Ch^{4N}$ for small $h$. Thus we can estimate the first term
\begin{equation}
\int_{\abs{h^2\rho^2-1}<\frac{1}{2}}(1+\rho^2)^{-2N}~\langle dE_\rho u, u\rangle\le Ch^{4N}\int_{\abs{h^2\rho^2-1}<\frac{1}{2}}~\langle dE_\rho u, u\rangle
\end{equation}
and bound it by $Ch^{4N}\|u\|^2$ as desired. 
\end{proof}

{The following lemma allows us to obtain a high-frequency unique continuation result.
\begin{lemma}[High-frequency unique continuation]\label{p:ucp}
Assume $Q$ is $m$-elliptic. Then 
\begin{equation}
(\img P_h)^\perp=\ker P_h^*=\{0\}
\end{equation}
and $P_h: H_{\frac{1}{2}}\rightarrow H_{-\frac{1}{2}}$ is invertible for all $h$ small. 
\end{lemma}
\begin{proof}
Assume $P_h^* u=(h^2P+ihQ^*Q-1)u=0$ for some $u$. This implies
\begin{equation}
\|hP^{\frac{1}{2}}u\|^2-\|u\|^2=0 , \ \|Qu\|_Y^2=0, \ Qu=0, \ (h^2P-1)u=0. 
\end{equation}
We then have
\begin{equation}
E_{h^{-1}}u=E_{h^{-1}}h^2Pu=u,
\end{equation}
and thus
\begin{equation}
\chi(h^2P)u=\int_{0}^{\infty}\chi(h^2\rho^2)~dE_\rho u=\chi(1) u. 
\end{equation}
Now apply Lemma \ref{p:cpterr} to see \eqref{eq:semiell} reduces to
\begin{equation}
\chi(1)^2\|u\|^2\le o(1)\|u\|^2.
\end{equation}
Thus uniformly for small $h>0$, this implies $u=0$ and $\ker P_h^*=\{0\}$. From \cite[Lemma 3.4]{kw23}, we know $P_h:H_{\frac{1}{2}}\rightarrow H_{-\frac{1}{2}}$ is Fredholm for all $h$ small. Thus $(\img P_h)^{\perp}=\ker P_h^*=\{0\}$ for all $h$ small. 
\end{proof}}

\begin{proof}[Proof of Proposition \ref{p:semires}]
Pair $P_h u$ with $u$ to observe that
\begin{equation}
\langle P_h u, u\rangle=\|hP^{\frac{1}{2}}u\|^2-\|u\|^2-ih\|Qu\|^2, 
\end{equation}
whose real and imaginary parts satisfy
\begin{equation}
\|hP^{\frac{1}{2}}u\|^2-\|u\|^2=\cre\langle P_h u, u\rangle, \ 
h\|Qu\|_Y^2=-\cim\langle P_h u, u\rangle. 
\end{equation}
We can estimate them by
\begin{gather}\label{eq:1}
\abs{\|hP^{\frac{1}{2}}u\|^2-\|u\|^2}\le 4\epsilon^{-2}\|P_hu\|^2+\epsilon^2 \|u\|^2,
\\
{\|Qu\|_Y^2\le 4\epsilon^{-2} h^{-2}\|P_hu\|^2+\epsilon^2 \|u\|^2.}
\end{gather}
The $m$-ellipticity \eqref{eq:semiell} implies
\begin{multline}
\|\chi(h^2P)u\|^2\le Cm^{-1}\|Qu\|_Y^2 +e(h)\|\Lambda^{-N}u\|^2 \\
\le C\epsilon^{-2}m^{-2} h^{-2}\|P_h u\|^2+\epsilon^2 \|u\|^2+{e(h)}\|\Lambda^{-N}u\|^2,
\end{multline}
where $e(h)=o(\min\{h^{-4N}, m(h^{-1})^{-1}h^{-1}\})$ as in \eqref{eq:semiell}. 
We apply Lemma \ref{p:cpterr} to estimate the compact error:
\begin{equation}
{e(h)}\|\Lambda^{-N}u\|^2\le o(1)\|u\|^2+e(h)\abs{\|hP^{\frac{1}{2}}u\|^2-\|u\|^2}\le C\epsilon^{-2}e(h)^2\|P_hu\|^2+\epsilon^2\|u\|,
\end{equation}
{for the last inequality of which we used a variant of \eqref{eq:1}. 
Noting $e(h)^2=o(m^{-2}h^{-2})$} by assumption, we have 
\begin{equation}
\|\chi(h^2P)u\|^2\le  C\epsilon^{-2}m^{-2} h^{-2}\|P_h u\|^2+\epsilon^2 \|u\|^2.
\end{equation}
Now since $\abs{\chi(s)}$ is uniformly bounded from below near $s=1$, the algebraic inequality
\begin{equation}
C^{-1}\le \abs{h^{2}\rho^2-1}+\abs{\chi(h^2\rho^2)}^2
\end{equation}
holds uniformly for all $h>0$, $\rho\in \mathbb{R}$. Thus
\begin{equation}
\|u\|^2=\int_0^{\infty} ~d\langle E_\rho u, u\rangle\le C\int_0^{\infty} \abs{h^2\rho^2-1}+\abs{\chi(h^2\rho^2)}^2~d\langle E_\rho u, u\rangle
\end{equation}
can be estimated by
\begin{equation}
\|u\|^2\le C\left(\abs{\|hP^{\frac{1}{2}}u\|^2-\|u\|^{2}}+\|\chi(h^2P)u\|^2\right)\le C(1+m^{-2}h^{-2})\|P_h u\|^2+\epsilon^2\|u\|^2. 
\end{equation}
We absorb the $\epsilon$-term to observe
\begin{equation}
\|P_h^{-1}\|_{\mathcal{L}(H)}\le C\max \{m^{-1}h^{-1}, 1\},
\end{equation}
as desired. 
\end{proof}

We now move on to prove the lower bound for $P_h^{-1}$. To do so, it is convenient to introduce the semiclassical interpolation spaces $H_s^{h}$. Define the semiclassical scaling operators
\begin{equation}
\Lambda^{s}_h u=\int_0^\infty (1+h^2\rho^2)^{s}\ dE_\rho(u),
\end{equation}
and interpolation spaces $H_s^h=\Lambda^{-s}_h(H)$ equipped with the norm $\|\cdot\|_{H_{s}^h}=\|\Lambda^s_h \cdot\|_{H}$. 
\begin{lemma}[Norm equivalence]\label{p:equiv}
For $s\ge 0$ and any $\chi$ Borel measurable on $(0,\infty)$ whose support is away from $0$, there exists $C>0$ such that
\begin{gather}
C^{-1}h^{2s}\|\Lambda^s u\|\le \|\Lambda^s_h u\|\le C\|\Lambda^s u\|,\\
\|\Lambda^s_h \chi(h^2P) u\|\le C h^{2s}\|\Lambda^s \chi(h^2P) u\|,\\
{h^{-2s}\|\Lambda^{-s}\chi(h^2P)u\|\le C\|\Lambda_h^{-s}\chi(h^2P)u\|},
\end{gather}
uniformly for all $u\in H_s$ and all $h>0$ small. 
\end{lemma}
\begin{proof}
It suffices to note the algebraic inequalities
\begin{gather}
C^{-1}h^{2s}(1+\rho^2)^s\le (1+h^2\rho^2)^s\le C(1+\rho^2)^s,\\
(1+h^2\rho^2)^s\chi(h^2\rho^2)\le C (h^2\rho^2)^s\chi(h^2\rho^2)\le Ch^{2s}(1+\rho^2)^s\chi(h^2\rho^2),\\
{h^{-2s}(1+\rho^2)^{-s}\chi(h^2\rho^2)\le C(h^2+h^2\rho^2)^{-s}\chi(h^2\rho^2)\le C(1+h^2\rho^2)^{-s}\chi(h^2\rho^2),}
\end{gather}
for the last two lines of which we used that $\chi$ is supported away from $0$. 
\end{proof}
We will use the following lemma to compare different norms of $P_h^{-1}$: it is a semiclassical version of \cite[Lemma 3.9]{kw23}. 
\begin{lemma}[Operator norm estimate%inequality
]\label{p:normineq}
Assume that $P_h : H_{\frac{1}{2}} \to H_{-\frac{1}{2}}$ is invertible for fixed $h > 0$. Then
\begin{equation} 
\|P_h^{-1}\|_{H\rightarrow H_{\frac{1}{2}}^h}\le C(1 +\|P_h^{-1}\|_{\mathcal{L}(H)}).
\end{equation}
Here $C$ does not depend on $h$. 
\end{lemma}
\begin{proof}
Given any $f\in H_{-\frac{1}{2}}$, there exists a unique $u\in H_{\frac{1}{2}}$ such that $P_h u=f$.
 Pair
\begin{equation}
\langle P_h u, u\rangle=\|hP^{\frac{1}{2}}u\|^2-\|u\|^2-ih\|Qu\|_Y^2,
\end{equation}
whose real part implies
\begin{equation}
\|hP^{\frac{1}{2}}u\|^2\le C\epsilon^{-1}{\|P_h u\|^2}+\|u\|^2+\epsilon {\|u\|^2}. 
\end{equation}
Note that, after absorption of $\epsilon$, we have
\begin{equation}
\|u\|_{H_{\frac{1}{2}}^h}^2=\|u\|^2+\|hP^{\frac{1}{2}}u\|^2\le C\|P_h u\|^2+C\|u\|^2,
\end{equation}
that is
\begin{equation}
\|P_h^{-1}f\|_{H_{\frac{1}{2}}^h}\le C(\|f\|_{H} +\|P_h^{-1}f\|_H),
\end{equation}
yielding the desired operator norm bound. 
\end{proof}

We now prove the lower bound for $P_h^{-1}$.
\begin{proposition}[Resolvent lower bound]\label{p:semilwb}
Let $m$ be a positive continuous function and let $Q$ is $m$-bounded, that is, {for some $\chi\in C^0([0, \infty))$ with $\chi(1)>0$,} there is $C>0$ such that
\begin{equation}\label{eq:semibdd}
\|Q^*Q\chi(h^2P) u\|_{H_{-\frac{1}{2}}^h}\le Cm(h^{-1})\| u\|_{H^h_{\frac{1}{2}}}. 
\end{equation}
uniformly for all $u\in H_{\frac{1}{2}}$ and all real $h\le C^{-1}$. Then there is $C>0$ such that
\begin{equation}\label{eq:2}
\|P_h^{-1}\|_{\mathcal{L}(H)}\ge C^{-1}(m(h^{-1}))^{-1}h^{-1},
\end{equation} 
along some sequence of $h\rightarrow 0$.
\end{proposition}

\begin{proof}
{If $P_h:H_{\frac{1}{2}}\rightarrow H_{-\frac{1}{2}}$ fails to be invertible on a sequence of $h\rightarrow 0$, then \eqref{eq:2} trivially holds due to \cite[Lemmata 3.6 and 3.9]{kw23}}. Without loss of generality we assume $P_h$ is invertible for all $h$ small. Since $P$ has compact resolvent, $P$ is unbounded and has discrete spectrum. There then exists a sequence of nontrivial $u_h\in H_{\frac{1}{2}}$ as $h\rightarrow 0$, such that $(h^2P-1)u_h=0$. Let $\chi\in C_c^\infty((0,\infty))$ be a cutoff with $\chi(1)=1$. Now we have $u_h=\chi(h^2P)u_h$ and
\begin{equation}
\|P_h^* u_h\|_{H^h_{-\frac{1}{2}}}=h\|Q^*Q\chi(h^2P)u_h\|_{H^h_{-\frac{1}{2}}}\le Cmh\|u_h\|_{H^h_{\frac{1}{2}}}\le Cmh\|u_h\|. 
\end{equation}
Recalling that $H_{\frac{1}{2}}^h$ and $H_{-\frac{1}{2}}^h$ are in duality, we then have
\begin{equation}
\|P_h^{-1}\|_{H\rightarrow H_{\frac{1}{2}}^h}=\|(P_h^*)^{-1}\|_{H_{-\frac{1}{2}}^h\rightarrow H}\ge C^{-1}m^{-1} h^{-1}.
\end{equation}
Apply the norm inequality in Lemma \ref{p:normineq} while noting $m$ is bounded from below to see
\begin{equation}
\|P_{h}^{-1}\|_{\mathcal{L}(H)}\ge C^{-1}m^{-1} h^{-1}
\end{equation}
as desired. 
\end{proof}

\begin{proof}[Proof of Theorems \ref{t:decay} and \ref{t:sharp}]
It remains to convert those resolvent estimates to upper and lower bounds for $\|(\mathcal{A}+i\lambda)^{-1}\|$. 
 We have the following characterisation from \cite[Lemma 3.7]{kw23}: for any $m$ bounded from above, 
\begin{equation}\label{l1}
\|P_h^{-1}\|_{\mathcal{L}(H)}\le C (m(h^{-1}))^{-1}h^{-1}
\end{equation}
uniformly for $h$ small is equivalent to 
\begin{equation}
\|(\mathcal{A}+i\lambda)^{-1}\|_{\mathcal{L}(\mathcal{H})}\le \frac{C}{m(\lambda)}, 
\end{equation}
as desired. Apply \cite[Lemma 3.9]{kw23} to obtain the lower bound of $(\mathcal{A}+i\lambda)^{-1}$.
\end{proof}

{\begin{proof}[Proof of Corollary \ref{t3}]
Theorem \ref{t:decay} already gives bounds on $(\Ac+i\lambda)^{-1}$. With \cite[Lemma 3.10]{kw23}, it remains to show that the unique continuation property holds to conclude the energy decay. That is, we want to show for any $\lambda>0$, $(P-\lambda^2)u=0$ implies $Qu\neq0$. When $(P-\lambda^2)u=0$, \cite[Lemma 3.3(4)]{kw23} implies $\Lambda^s u\neq 0$. Now \eqref{eq:4} implies $Qu\neq 0$ as desired. 
\end{proof}}

We conclude the note with a lemma to conveniently turn classical ellipticity and boundedness of $Q$ into weak ellipticity \eqref{eq:semiell} and weak boundedness \eqref{eq:semibdd}.
\begin{theorem}[Classical estimates]\label{p:classell}
{Let $h=\lambda^{-1}$.} The following are true:
\begin{enumerate}[wide]
\item If $Q$ is classically elliptic with respect to $\Lambda^s$ for $s\in (-\infty, \frac{1}{2}]$, that is, there are $C, N>0$ such that $-N<s$ and 
\begin{equation}
\|\Lambda^s u\|^2\le C\|Qu\|_{Y}^2+C\|\Lambda^{-N}u\|^2,
\end{equation}
for all $u\in H_{\frac{1}{2}}$. Then $Q$ is ${\lambda^{4s}}$-elliptic: for some $\chi\in C^\infty([0,\infty))$ with $\chi(1)>0$, 
\begin{equation}
{h^{-4s}}\|\chi(h^{2}P) u\|^2\le C\|Qu\|_Y^{2}+C\|\Lambda^{-N}u\|^2.
\end{equation}
Note here the compact error is of desired size {$C=o(h^{-4s-4N})=o(h^{-1})$}.
\item If $Q$ is bounded by $\Lambda^{s}$ for $s\in [0, \frac{1}{2}]$, that is, there is $C>0$ such that
\begin{equation}
\|Q u\|_Y^2\le \|\Lambda^{s}u\|_H^2
\end{equation}
for all $u\in H_{s}$. Then $Q$ is {$\lambda^{4s}$}-bounded, that is,
\begin{equation}
\|\Lambda_h^{-\frac{1}{2}}Q^*Q\chi(h^{2}P)u\|\le {h^{-4s}}C\|\Lambda_h^{\frac{1}{2}}u\|. 
\end{equation}

\item If $Q$ is bounded by $\Lambda^s$ for $s\in (-\infty, 0)$, then $Q$ is {$\lambda^{2s}$}-bounded. If additionally $Q^*Q$ is \emph{microlocal}, in the sense that there are cutoffs $\phi, \psi\in C^\infty[0,\infty)$ with $\phi(s)\equiv 1$ near $s=0$, $\phi(s)\equiv 1$ near $s=1$, $\phi\psi=0$ such that for all $h$ small and all $u\in H$, 
\begin{equation}\label{eq:microlocal}
\|\phi(h^2P)Q^*Q\psi(h^2P) u\|_H\le Ch^{-4s}\|u\|_H .
\end{equation}
Then $Q$ is $\lambda^{4s}$-bounded.

\item When $s\in(-\infty, 0)$, if $Q^*Q$ is bounded by $\Lambda^{2s}$ on $H_{2s}$, that is
\begin{equation}
\|Q^*Qu\|_{H}\le C\|\Lambda^{2s} u\|_H 
\end{equation}
uniformly for $u\in H_{2s}$. {Then $Q^*Q$ is microlocal as in (3), and} $Q$ is {$\lambda^{4s}$}-bounded. 
\end{enumerate}
\end{theorem}

\begin{proof}
{We again use the semiclassicalisation $h=\lambda^{-1}$}.

1. Assume $Q$ is classically elliptic with respect to $\Lambda^{s}$. When $0\le s\le \frac{1}{2}$, we have for large $\lambda$, that
\begin{equation}
\|(h^2P)^s\|^2\le Ch^{4s}\|\Lambda^{s} u\|^2\le Ch^{4s}\|Q u\|^2_{Y}+Ch^{4s}\|\Lambda^{-N} u\|^2
\end{equation}
as desired. When $s<0$, consider a nonnegative cutoff $\psi\in C^\infty_c((0,\infty))$ with $\psi(1)=1$. Then from Lemma \ref{p:equiv}
\begin{equation}
\|\psi(h^2 P) \Lambda_h^su\|^2\le Ch^{4s}\|\Lambda^s u\|^2\le Ch^{4s}\|Q u\|^2_{Y}+Ch^{4s}\|\Lambda^{-N} u\|^2
\end{equation}
as desired. 

2. Assume $Q$ is bounded by $\Lambda^s$. Note then $\Lambda^{-s}Q^*: Y\rightarrow H$ is bounded and we have
\begin{equation}
\|\Lambda^{-s}Q^*Qu\|\le C\|\Lambda^s u\|.
\end{equation}

2a. When $0\le s\le \frac{1}{2}$, with Lemma \ref{p:equiv}, we have 
\begin{equation}
\|\Lambda_h^{-\frac{1}{2}}Q^*Q u\|\le C\|\Lambda_h^{-s}Q^*Q u\|\le Ch^{-2s}\|\Lambda^{-s} Q^*Q u\|\le Ch^{-2s}\|\Lambda^s u\|\le Ch^{-4s}\|\Lambda_h^s u\|
\end{equation}
as desired. 

2b. When $s<0$, consider a nonnegative cutoff $\psi\in C_c^\infty((0,\infty))$ with $\psi(1)=1$. Note $Q^*: Y\rightarrow H$ is bounded and we have with Lemma \ref{p:equiv},
\begin{equation}
\|\Lambda_h^{-\frac{1}{2}}Q^*Q\psi(h^2P) u\|\le C\|Q^*Q \psi(h^2P)u\|\le C\|\Lambda^s \psi(h^2P)u\|\le C h^{-2s}\|\Lambda^s_h u\|
\end{equation}
showing that $Q$ is $\lambda^{2s}$-bounded as desired. 

2c. When $s<0$, assume additionally there are cutoffs $\phi, \psi\in C^0([0,\infty))$ with $\phi(s)\equiv 1$ near $s=0$, $\psi(s)\equiv 1$ near $s=1$, $\phi\psi=0$ such that for all $h$ small and all $u\in H$,
\begin{equation}
\|\phi(h^2P)Q^*Q\psi(h^2P) u\|_H\le Ch^{-4s}\|u\|_H.
\end{equation}
Now consider
\begin{equation}
\Lambda_h^{-\frac{1}{2}}Q^*Q\psi(h^2P) u=\Lambda_h^{-\frac{1}{2}}(1-\phi)(h^2P)Q^*Q\psi(h^2P) u+\Lambda_h^{-\frac{1}{2}}\phi(h^2P)Q^*Q\psi(h^2P) u,
\end{equation}
the second term has the size $Ch^{-4s}\|u\|$ from the assumption. We can estimate the first term. Note $1-\phi$ is supported away from $0$. Apply Lemma \ref{p:equiv} to see
\begin{multline}
\|\Lambda_h^{-\frac{1}{2}}(1-\phi)(h^2P)Q^*Q\psi(h^2P) u\|\le \|\Lambda_h^{-s}(1-\phi)(h^2P)Q^*Q\psi(h^2P) u\|\\
\le Ch^{-2s}\|\Lambda^{-s}(1-\phi)(h^2P)Q^*Q\psi(h^2P) u\|\le Ch^{-2s}\|Q\psi(h^2P) u\|_Y\\
\le Ch^{-2s}\|\Lambda^s\psi(h^2P) u\|
\le Ch^{-4s}\|u\|. 
\end{multline}
We then have
\begin{equation}
\|\Lambda_h^{-\frac{1}{2}}Q^*Q\psi(h^2P) u\|\le Ch^{-4s}\|u\|
\end{equation}
yielding the desired $\lambda^{4s}$-boundedness. 

3. Assume $s<0$ and $Q^*Q$ is bounded by $\Lambda^{2s}$. Pick any cutoffs $\phi,\psi$ as described in Step 2c. Then
\begin{equation}
\|\phi(h^2P)Q^*Q\psi(h^2P) u\|\le \|Q^*Q\psi(h^2P) u\|\le C\|\Lambda^{2s}\psi(h^{2}P)u\|\le Ch^{-4s}\|u\|
\end{equation}
uniformly for all $H$ as desired. Thus $Q^*Q$ is microlocal and we apply (3) to see $Q$ is $\lambda^{4s}$-bounded. 
\end{proof}

\begin{remark}
In the case $s < 0$, Theorem \ref{p:classell}(3) may be suboptimal without the microlocality of $Q^*Q$. The microlocality \eqref{eq:microlocal} forbids the communication between zero sections $\{\rho=0\}$ and semiclassical characteristics $\{\rho= h^{-1}\}$. Note that if $Q^*Q=f(P)$ is indeed within the functional calculus of $P$, then $Q^*Q$ is automatically microlocal:
\begin{equation}\label{eq:3}
\phi(h^2P)Q^*Q\psi(h^2P)u=\int_0^{\infty} \phi(\rho^2)f(\rho^2)\psi(\rho^2)~dE_\rho u=0,
\end{equation}
whenever $\phi\psi=0$. When $P$ is a self-adjoint nonnegative pseudodifferential operator of positive order on a compact smooth manifold without boundary, and $Q^*Q$ is the multiplication by a smooth function, the microlocality also holds with the right of \eqref{eq:3} replaced by $h^{\infty}$: see \cite[equation {(E.2.5)}]{dz19}. 
In practice, it is easier to check the classical boundedness of $Q^*Q$ and use Theorem \ref{p:classell}(4) instead: see Examples \ref{eg:2}, \ref{p:egpdo}.
\end{remark}

\bibliographystyle{alpha}
\bibliography{Robib,addbib}

\end{document}